%% file: numberof_Kdegree_partialresults.tex
\documentclass[options]{amsart}
 \usepackage{thm-restate}
 
\makeatletter
\@ifundefined{newcounteralias}{}{%
  \renewcommand\thmt@autorefsetup{%
    \@xa\def\csname\thmt@envname autorefname
      \@xa\endcsname\@xa{\thmt@thmname}%
  }%
}
\makeatother

\usepackage{latexsym,ifthen,amssymb}
\usepackage[toc,page,title,titletoc,header]{appendix}
\usepackage{enumerate}
\usepackage{graphicx}
\usepackage{bm}
\usepackage{subfigure}
\usepackage{float}

\usepackage{tikz-cd}
\usepackage{caption}

\usepackage{multirow}
\usepackage{mathrsfs}

\usepackage{amsthm}
\usepackage{hyperref}
\usepackage{cleveref}
\usepackage{bm}
\newtheorem{theorem}{Theorem}[section]
\newtheorem{lemma}[theorem]{Lemma}
\newtheorem{claim}[theorem]{Claim}
\newtheorem{question}[theorem]{Question}
\newtheorem{proposition}[theorem]{Proposition}

\numberwithin{equation}{section}
\theoremstyle{remark}
\newtheorem{remark}[theorem]{Remark}

\Crefname{question}{Question}{Question}
 \crefname{claim}{claim}{claims}
\Crefname{claim}{Claim}{Claims}

\input{marcro.tex}

\def\Y{Y}
\def\D{\Delta}
\title{Cardinality of the K-cones below a non infinitely often $K$-trivial}

 \author{George Barmpalias    }
 \address{State Key Lab of Computer Science, Institute of Software, Chinese Academy of Sciences, Beijing, China. }
\email{barmpalias@gmail.com }

\author{Lu Liu}
\address{School of Mathematics and Statistics,
Central South University,
City Changsha, Hunan Province,
China. 410083}
\thanks{Lu Liu is supported by  NSFC grant 12471002.}
\email{g.jiayi.liu@gmail.com}

\begin{document}
\maketitle

\begin{abstract}

Given a real $x\in 2^\omega$, how many reals are there 
that are simpler (or equally simple) to describe than $x$.
With ``simpler to describe" characterized by $K$-complexity 
of initial segments (hence we call such $y$ with $y\leq_K x$ a solution
to $x$), if $x$ is infinitely often $K$-trivial
(hence i.o. $K$-trivial),
then a simple counting argument shows that the number of solutions to $x$
is countable. 
If $x$ is non  i.o. $K$-trivial, then must it admit uncountably 
many solutions?
Motivated by this question, we show that:
\begin{enumerate}
\item  There is a non i.o. $K$-trivial $x$ such that most of its solution
is i.o. $K$-trivial. i.e., Only countably many solution to $x$
 are non i.o. $K$-trivial. 
 \item There is a minimal pair of K-degree
 (meaning their common solutions are all $K$-trivial) 
 so that each of the pair is non i.o. $K$-trivial,
 generalizing a result of Csima and Montalb\'an
 (where the pair is not required to be non i.o. $K$-trivial but non $K$-trivial).

\item There is a non i.o. $K$-trivial $x$ whose $\Delta_K$ complexity grows arbitrarily slow
(in a weak sense: $\inf_{m\geq n}\Delta_K(x\uhr m)\leq h(n)$ with $h$
a given
unbounded and nondecreasing function) such that $x$ admits uncountably many solutions.

\item An ample excess lemma holds for non i.o. $K$-trivial. Namely:
$$
\sum_n 2^{-(K(x\uhr n)- \log_2 n)}<\infty.
$$ for all non i.o. $K$-trivial $x$. Actually, this even holds
for non i.o. $b$-$K$-trivial with $b$
only depends on the universal prefix free machine.
The same argument also gives a constant complexity return phenomenon for $K$-trivial
where this constant only depends on the universal prefix free machine.
\end{enumerate}

\end{abstract}

\input{numberof_Kdegree_introduction}

\input{smallnessofsolutions}

\input{arbitraryslowgrowth}

\input{constantdropofKtrivial}

\input{acknowledgement}

\end{document}

%% file: marcro.tex
\def\mcal{\mathcal}
\def\h{\hat}
\def\t{\tilde}

\def\m{m}

\def\uhr{\upharpoonright}

%% file: numberof_Kdegree_introduction.tex
\section{Introduction}

Given a real $x$, how many reals $y$ are simpler than $x$?
We measure simplicity by the lengths of shortest descriptions of
initial segments. For a fixed optimal prefix-free machine $U$, let
$K(\sigma)$ be the length of a shortest program producing the finite
binary string $\sigma$. This approach to complexity originates in
the work of Solomonoff and Kolmogorov; the prefix-free theory was
developed by Levin and Chaitin
\cite{Solomonoff64,Kolmogorov65,Chaitin75}.
Changing the optimal machine changes $K$ by at most an additive
constant. Thus, for a real $x\in 2^\omega$, the function
$n\mapsto K(x\uhr n)$, where $x\uhr n$ consists of its first $n$
bits, gives a machine-independent comparison of descriptive
complexity up to such constants.

Following Downey, Hirschfeldt and LaForte \cite{DHL04}, write
$y\leq_K x$ if $$K(y\uhr n)\leq K(x\uhr n)+O(1)$$ for all $n$.
 Our question concerns the
cardinality of the lower $K$-cone $$\{y:y\leq_K x\}.$$
Mutual reducibility defines $=_K$, and its equivalence classes
are the $K$-degrees. This comparison retains information about
every initial-segment length, including fluctuations that an
asymptotic rate of complexity would lose.

The least possible complexity is $K(n)+O(1)$: the length $n$
can be recovered from any string of length $n$, and a computable
real has initial segments describable from their lengths. A real
$x$ is \emph{$K$-trivial} if $$K(x\uhr n)\leq K(n)+O(1)$$ for all
$n$. Solovay showed that such reals need not be computable
\cite{DHNS03}. They form the least $K$-degree, so no real is
strictly simpler than a $K$-trivial real in this ordering.
Their descriptive simplicity also entails computational weakness.
For example, every $K$-trivial real is jump-traceable: the values
of a universal partial function computable from it can be covered
by uniformly c.e. finite sets whose sizes have a computable bound
\cite{Nies05,BDG09}.

There are only countably many $K$-trivial reals. The counting
behind this fact applies more generally. Uniformly in $n$, there
are $O(2^d)$ strings $\sigma$ of length $n$ with
$K(\sigma)\leq K(n)+d$ \cite{MY11}.
Suppose that $K(x\uhr n)\leq K(n)+b$ at infinitely many lengths,
for a fixed $b$; such an $x$ is called \emph{infinitely often
$K$-trivial}, or \emph{i.o.\ $K$-trivial}.
For each fixed comparison constant $c$, the initial segments of
all $y$ satisfying $K(y\uhr n)\leq K(x\uhr n)+c$ have boundedly
many possibilities at these common lengths. Hence there are only
finitely many such $y$. Taking the union over $c$ gives a
countable lower cone \cite{BV11}. We study the converse question,
conjectured by Barmpalias \cite[Section~3.5]{Barmpalias13}:
\begin{question}\label{numKques}
Does every non i.o.\ $K$-trivial real $x$ have uncountably many
reals $y\leq_K x$?
\end{question}
\noindent Let $\Delta_K(\rho):= K(\rho)-K(|\rho|)$,
 The hypothesis says precisely that $$\Delta_K(x\uhr n)\to\infty.$$

The difficulty for \Cref{numKques} is that initial-segment complexity functions cannot
be chosen freely. For example, $K(z\uhr n)=n+O(1)$ is
impossible by the ample excess lemma of Miller and Yu \cite{MY08}.
Also, if $z_0\oplus z_1$ is Martin-L\"of random, the two reals
have no common $K$-upper bound \cite{MY08}; even the pointwise
maximum of two realized functions need not be realizable.
For the lower $K$-complexity zone,
Csima and Montalb\'an \cite{CsimaMontalban2006}
constructed a nondecreasing unbounded function $h$ such that
$K(y\uhr n)\leq K(n)+h(n)+O(1)$ for every $n$ forces $y$ to be
$K$-trivial. Thus the numerical bound $K(n)+h(n)$ rises
unboundedly above $K(n)$, yet admits only countably many reals
below it. It cannot itself equal $K(x\uhr n)+O(1)$ for any real
$x$. Our question therefore depends on the upper bound being
realized by a real, as well as on its distance above $K(n)$.
There are nevertheless broad realization theorems. A consequence
of Staiger's results \cite[Theorem~3, Proposition~1 and Remark~1]{Staiger12}
is that, if $f:\omega\to\mathbb R$ satisfies
\[
  0<a\leq f(n+1)-f(n)\leq b<1
\]
for some constants $a,b$ and every $n$, then
$K(z\uhr n)=f(n)+O(1)$ for some real $z$.
No computability assumption on $f$ is required. At the endpoints,
restrictions reappear. 
Recently Liu
\cite{Liu26} proved that     every
real has a strictly more complex real above it
(i.e., There is no maximal $K$-degree).

Regarding \Cref{numKques}, for lower cones, a general sufficient condition due to Baartse and Barmpalias \cite{BB13}
   is: $\inf_{m\geq n}\Delta_K(x\uhr m)$ dominates 
some $\emptyset'$-computable nondecreasing unbounded function.
Under this condition the lower cone contains a perfect set.
This reaches far below Martin-L\"of randomness, which requires
$$K(x\uhr n)\geq n-O(1),$$ but leaves a gap between this sufficient
condition and divergence above $K(n)$ alone.

In contrast to the ``no maximal $K$-degree" result,
one can even ask if there is a minimal non zero $K$-degree \cite[Question~8]{Barmpalias13}.
That is, a $K$-degree so that only $K$-trivial reals lies strictly below it.
This is not even known for non i.o. $K$-trivial reals.  
 This is a structural question
distinct from counting the whole lower cone, which also includes
all reals of the same $K$-degree as $x$.

\subsection{Main results}

Fix a universal prefix free machine $U$
and denote $K_U$ by $K$.
For    $x,y\in 2^\omega\cup 2^{<\omega}$
with identical length
and $c\in\omega$,
we write $y\leq_K^c x$
iff for every $n\leq |x|$, 
$$K(y\uhr n)\leq K(x\uhr n)+c.$$
In which case,
we   say $y$ is a $c$-\emph{solution} to $x$;
we say $y$ is a solution to $x$ iff it is a $c$-solution to $x$ for some $c\in\omega$
(that is, $y\leq_K x$);
let $\Y_c(x)$ denote the set of $c$-solution to $x$.

Our first result is a non-free-choice result for $K$ complexity function.
Which says that there is some non i.o. $K$-trivial real, 
while most of reals K-below it are i.o. $K$-trivial.
The proof of this result allows the 
 non i.o. $K$-trivial to have uncountable lower K-cones
 (see \Cref{numKrem2}).
\begin{restatable}{proposition}{numKpropfinitenoniosolution}
\label{numKpropfinitenoniosolution}
There exists a non i.o. $K$-trivial $x\in 2^\omega$ such that 
for every $c\in\omega$, there are finitely many $c$-solutions
 of $x$ that is non i.o. $K$-trivial.
 Thus,
there are countably many reals $y$ with $y\leq_K x$ such that 
$y$ is  non i.o. $K$-trivial. 

\end{restatable}

In particular, the $K$-degree of this $x$ (and all its 
non i.o. $K$-trivial solution) is countable.
Countable non i.o.\ $K$-trivial degrees already occur among
Martin-L\"of random reals \cite{Miller09}.
 The question about the whole cone remains unresolved.

Csima and Montalb\'an \cite{CsimaMontalban2006}
proved that there are two non $K$-trivial reals $x_0,x_1$
such that every common solution to $x_0,x_1$ is $K$-trivial.
The following result strengthens
their result by allowing 
the two reals to be non i.o. $K$-trivial.

\begin{restatable}{proposition}{numKpropintersectionoftwo}
\label{numKpropintersectionoftwo}

There are two non i.o. $K$-trivial  $x_0,x_1$
such that for every $c\in\omega$, $\Y_c(x_0)\cap \Y_c(x_1)$ is a finite
set of $K$-trivial.
Thus, there are countably many reals $y$ such that $y\leq_K x_0,x_1$.
\end{restatable}
 
Baartse and Barmpalias \cite{BB13}
shows that for every $\Delta_2^0$,
nondecreasing unbounded function $h$,
there are uncountably many reals $y$
whose $\Delta_K$ complexity function is dominated 
by $h$.
The following proposition says that 
there is a non i.o. $K$-trivial whose $\Delta_K$ complexity
grows arbitrary slow (in a weak sense)  that
admits uncountably many solutions.

\begin{restatable}{proposition}{numKpropslowgrowth}
\label{numKpropslowgrowth}
For any sufficiently large
(depending on the universal prefix free machine) $c\in\omega$,
any non decreasing unbounded function $h\in\omega^\omega$,
there exists a non i.o. $K$-trivial $x\in 2^\omega$
with its $\Delta_K$ complexity controlled by $h$
in the sense: $\inf_{m\geq n} \Delta_K(x\uhr m)\leq h(n)$
for all but finitely many $n\in\omega$,
such that $\Y_c(x)$ is uncountable.

\end{restatable}

Miller and Yu \cite{MY08} 
discovered the famous ample excess lemma, which says that
$$
\sum_n2^{-(K(x\uhr n)-n)}<\infty
$$
for all 1-random real $x$.
Jason Rute  (see \cite[slide~17]{Simpson14} for more details)
generalized ample excess lemma
 which says that if real $x$ is strongly $f$-random,
 then $$\sum_n 2^{-(K(x\uhr n)-f(x\uhr n))}<\infty.$$
Here $f:2^{<\omega}\to\mathbb R$ is computable.
We have an analog of ample excess lemma for 
non i.o. $K$-trivial. Actually, it even holds
for real that is not i.o. $K$-trivial witnessed by 
a $O_U(1)$ constant.
For a real $x\in 2^\omega$ and $a\in \omega$, we say $x$ is 
\emph{i.o. $a$-$K$-trivial}
iff for infinitely many $n\in\omega$, $\Delta_K(x\uhr n)\leq a$.

\begin{restatable}{theorem}{numKampleexcess}
\label{numKampleexcess}
Let $b\in\omega$ be sufficiently large depending on $U$.
For every non i.o. $b$-$K$-trivial real $x$,
we have
$$
\sum_{n\geq 1} 2^{-(K(x\uhr n)-\log_2 n)}<\infty.
$$
In particular, this holds for all non i.o. $K$-trivial.
\end{restatable}

The rest of this paper will prove the above mentioned main results.
We prove \Cref{numKpropfinitenoniosolution}
and \Cref{numKpropintersectionoftwo} in \Cref{numKsecsmall};
\Cref{numKpropslowgrowth} in   \Cref{numKsecslowgrowth};
and \Cref{numKampleexcess} in   \Cref{numKsecampleexcess}.
During these proof, we delay the proof of technical results
(usually these are specific properties concerning $K$ complexity function)
to \Cref{numKtechnicalresults}.
We end up this section by giving notations.

\subsection{Summary of notations.}
\ \\

 \noindent\textbf{Strings.}
 For $\sigma,\rho\in 2^{<\omega}\cup 2^\omega$, 
we write $\sigma\succeq\rho$ if $\sigma$ extends $\rho$;
 write $\sigma\uhr n$ for the string $\sigma(0)\cdots\sigma(n-1)$ 
 (if $n>|\sigma|$, then $\sigma\uhr n:= \sigma$);
write $|\sigma|$ for the length of $\sigma$
let $[\sigma]^\preceq $ be the set of strings extending $\sigma$;
let $[\sigma] := \{x\in 2^\omega:x\succeq\sigma\}$;
we  write $\sigma\rho$
or $\sigma\cdot\rho$  for the concatenation of $\sigma$ and $\rho$.
For a set $A\subseteq 2^{<\omega}$, let $[A]^\preceq := \cup_{\sigma\in A}[\sigma]^\preceq$;
and  $[A]: = \cup_{\sigma\in A}[\sigma]$.

\noindent\textbf{Measure.}
For a set $Y\subseteq 2^\omega$, let $\m(Y)$ denote 
the Lebesgue measure of $Y$.
For a set $A\subseteq 2^{<\omega}$, the measure of $A$,
denoted as $\m(A)$, refers to
the Lebesgue measure of $[A]$;
for $\sigma\in 2^{<\omega}$, we write $\m(\sigma)$ for $\m(\{\sigma\})$;
we write $\m(Y|X)$ for $\m(Y\cap X)/\m(X)$.

\noindent\textbf{Big O small o}.
 We write $O(1)$ for  absolute constant;
 write $O_v(1)$ for a constant depending on object $v$.
For real variables $u,v,w$,
 we write $v\leq O_u(w)$ iff $v\leq O_u(1)\cdot w$
 (for all $u,v,w$).
 We say:
for every $v$  \emph{sufficiently large depending on $u$},
$\psi(v,u)$ holds
iff there is a function $h$ such that
for every $u$ and every $v\geq h(u)$,
$\psi(v,u)$ holds.

\noindent\textbf{K-complexity and others.}
Let $\Delta_K(\rho):= K(\rho)-K(|\rho|)$.
For $a\in\omega$, let
  $$
 \D_{\leq a}:= \{\sigma\in 2^{<\omega}: \Delta_K(\sigma)\leq a\},
 $$
For a set $X\subseteq 2^\omega\cup 2^{<\omega}$, let $\Y_c(X) $ denote the set of 
common $c$-solutions to elements of $X$.
i.e., 
$$
\Y_c(X):= \{y\in 2^\omega\cup 2^{<\omega}: y\in \Y_c(x)\text{ for all }x\in X\}.
$$


%% file: smallnessofsolutions.tex
\section{Smallness of the set of solutions}
\label{numKsecsmall}
Our target in this section is \Cref{numKpropfinitenoniosolution}
and   \Cref{numKpropintersectionoftwo}.
We firstly note that to construct a non i.o. $K$-trivial, it 
is sufficient and necessary to consider the following type of condition.
A \emph{condition} $p$ is a pair $(\rho,b)$
such that 
$\rho\in 2^{<\omega}$ and $b\in \omega$.
A condition is seen as a set of candidates of the reals $x$ we are constructing.
i.e., Those reals whose $\Delta_K$ complexity is above $b$ over level $|\rho|$:
$$
\big\{x\in [\rho] : \Delta_K(x\uhr n)\geq b\text{ for all }n\geq |\rho|\big\}.
$$
This set is denoted as $[\rho,b]$ (or $[p]$), which is clearly a closed set
and so it give rise to a pruned tree (rooted at $\rho$).
This tree 
 is denoted by   $(\rho,b)$ (or $p$).
 
 The reason
we consider   conditions is obvious, every non i.o. $K$-trivial $x$ must be produced by
a sequence of conditions $(\rho_s,b_s)_{s\in\omega}$ (where $\rho_s\prec\rho_{s+1}$ and 
$b_s<b_{s+1}$) in the sense that $x = \cup_s \rho_s$.
 
We refer to $b$ as the \emph{threshold} (of  $p$) and 
$\rho$ as the \emph{stem} (of   $p$).
We write $\Y_c(\rho,b)$ (or $\Y_c(p)$) for $\Y_c([\rho,b])$.
A condition $p$ being \emph{nonempty} means $[p]\ne\emptyset$.
A condition $(\h\rho,\h b)$ \emph{extends} a condition $(\rho,b)$
(also written as $(\h\rho,\h b)\leq (\rho,b)$) iff
$\h\rho\in (\rho,b)$ and $\h b\geq b$.
We have the following existence of condition extension fact.
\begin{restatable}{claim}{existenceofextension}\label{numKclaimext}
Let   condition $p=(\rho,b)$  so that $b$ is sufficiently large (depending on $U$). 
\begin{enumerate}
\item We have $$\m([p]|\rho)> 1-2^{-(\Delta_K(\rho)-2b-2\log_2 b)}.$$
Thus, 
if, in addition, $\Delta_K(\rho)\geq 3b $, then $m([p])>0$.
\item If  $p$ is nonempty, then $\m([p])>0$.
\item For every $\h b\geq b$, every $n\in\omega$,
there exists a nonempty extension $(\h\rho,\h b)$
of $(\rho,b)$ such that $|\h\rho|>n$.

 \end{enumerate}
\end{restatable}
The proof of \Cref{numKclaimext} is delayed to Section \ref{numKsecpositive}.

\subsection{Countable non i.o. $K$-trivial solutions}

In this section, we construct a non i.o. $K$-trivial $x\in 2^\omega$
such that there are countably many solutions of $x$ that are not i.o. $K$-trivial.
i.e., We prove
\numKpropfinitenoniosolution*

\noindent\textbf{Observations.}
We observe the following.
Let $b,c\in\omega$
and $ 2^b$ many  nonempty conditions
$(\rho_i,b)_{i<2^b}$ with $\rho_i\in 2^n$ (for some $n$ for all $i<2^b$) so that
\footnote{For $c$ sufficiently large, it is impossible 
that $\rho_i,\rho_j$ has no common solution. 
But this does not affect how we use the idea to control the size of non common solutions.}
\begin{align}\label{numKeqobserve1}
&\text{for each $i\ne j$, there is no (common) $c$-solutions to
 $\{\rho_i,\rho_j\}$
  and }\\ \nonumber
  &\text{$\Delta_K(\rho_i) =b+O_U(1)$.}
 \end{align}
 Then we have:
 \begin{align}\label{numKeqobserve2}
\text{ there is some $i<2^b$ such that $\rho_i$ admits $O_U(2^c)$ many solutions.}
 \end{align}
 To see this, note that for every $n,a\in\omega$,
recall notation
 \def\D{\Delta}
  $$
 \D_{\leq a}:= \{\sigma\in 2^{<\omega}: \Delta_K(\sigma)\leq a\},
 $$
it is well known that 
  $|2^n\cap \D_{\leq a}|\leq O_U(2^a)$
  (see  \cite{MY11}). 
 Thus, 
 since $$\cup_{i<2^b} \Y_c(\rho_i)\subseteq 2^n\cap \D_{\leq b+c+O_U(1)},$$
 there must be some $i<2^b$ so that $$|\Y_c(\rho_i)|\leq O_U(2^{b+c})/2^b = O_U(2^c).$$
  This  observation 
 indicates that $\Y_c(x)$ consists of two parts, 
\begin{enumerate}
\item The closure of common $c$-solutions of each condition;
\item Other solutions;
\end{enumerate}
and the second part could be controlled to be finite.

On the other hand, we will use the following properties of $K$ 
to entail that, other than finitely many exceptions, 
all other reals in the first part is i.o. $K$-trivial.
The first property says that every common $c$-solution of a condition is a $K$-trivial;
and the second property says that for every $K$-trivial, the $\Delta_K$ complexity of 
it drops below a constant (that only depends on $U$) infinitely often.
\begin{restatable}{lemma}{numKlemcommonsolutionisKtrivial}\label{numKlemcommonsolutionisKtrivial}
Let $c\in\omega$ and let nonempty condition $p$.
For every $y\in \Y_c(p)$, $y$ is $K$-trivial.
\end{restatable} 
The proof of \Cref{numKlemcommonsolutionisKtrivial} is
 delayed to   Section \ref{numKseccomplexityofcommonsolution}.
\begin{restatable}{proposition}{numKpropconstantreturnofKtrivial}\label{numKpropconstantreturnofKtrivial}
For every $K$-trivial $x\in 2^\omega$, we have 
$\Delta_K(x\uhr n)\leq O_U(1)$ for infinitely many $n\in\omega$.
\end{restatable}
\begin{proof}
We will prove a stronger version of 
 \Cref{numKpropconstantreturnofKtrivial},
namely \Cref{numKpropconstantreturn2}.
Clearly  \Cref{numKpropconstantreturnofKtrivial} follows directly from 
\Cref{numKpropconstantreturn2}.
\end{proof}

 The constant drop of $\Delta_K$ complexity of a $K$-trivial is interesting in its own.
We firstly realize the observation (\ref{numKeqobserve2})
to control the size of ``other" $c$-solutions.
 \begin{remark}
 \Cref{numKpropconstantreturnofKtrivial}
also implies that 
members of a condition with sufficiently large (depending on $U$) threshold  
cannot be $K$-trivial (but could be i.o. $K$-trivial).
 \end{remark}

\ \\

\noindent\textbf{Controlling the ``other" $c$-solutions.}

\def\maximal{maximal}
The goal of this subsection is \Cref{numKlem6}.
First, we count the number of common solutions.
To this end, we establish a property of $K$ that is stronger than what is needed now 
but its full content will be 
used later.

\begin{restatable}{lemma}{numKsimultaneousdrop}\label{numKsimultaneousdrop}
Let nonempty conditions $(\rho_i,b)_{i<2^b}$ so that 
$b$ is sufficiently large depending on $U$.
Then there are infinitely many $n\in\omega$ such that
$$2^n\cap \D_{\leq b+O_U(1)}\cap (\rho_i,b)\ne\emptyset$$
for all $i<2^b$.
\end{restatable}
\noindent The proof of \Cref{numKsimultaneousdrop} is delayed to Section \ref{numKseccomplexityofcommonsolution}.
By \Cref{numKsimultaneousdrop},
we obviously have, for every $b,c\in\omega$ and every nonempty condition $(\rho,b)$:
\begin{align}\label{numKeqcountingcommonsolution}
 |\Y_c(\rho,b)|\leq O_U(2^{b+c}). 
\end{align}

Given a nonempty condition $p$, we   want the common $c$-solution of $p$
to be stable:
a condition $p$ is \emph{$c$-\maximal}\  iff
the set of common $c$-solutions
cannot be enlarged by extending $p$ to some condition with identical threshold.
i.e.,
For every nonempty extension $q$ of $p$
 with identical threshold,
we have $\Y_c(q) = \Y_c(p)$.
Thus, by (\ref{numKeqcountingcommonsolution}), it's easy to see that for every $c\in\omega$,
\begin{align}\label{numKeqmaximal}
&\text{Every nonempty condition admits an extension}\\ \nonumber
&\text{with identical threshold that is $c$-\maximal.}
\end{align}

In order to take advantage of (\ref{numKeqobserve2})
to control the number of  solutions,
we need to fulfill its hypothesis, namely (\ref{numKeqobserve1}),
cutting off as many common solutions as possible.

\begin{lemma}\label{numKlem5}
Let $ c\in\omega, \sigma\in 2^{<\omega}$ and
 nonempty $c$-\maximal\ conditions $(p_i=(\rho_i,b_i): i<2)$.
Suppose 
 $[\sigma]\cap \Y_c(\rho_1,b_1)=\emptyset$.
Then there are nonempty extension $(\h\rho_i,b_i)$ of $(\rho_i,b_i)$
with $|\h\rho_0| = |\h\rho_1|$
such that
 $[\sigma]^\preceq\cap \Y_c(\{\h\rho_0,\h\rho_1\})=\emptyset$.

\end{lemma}
\begin{proof}
Firstly, we note that
\begin{claim}\label{numKclaim1}
For  every $k\in\omega$,
every sufficiently large
(depending on $\rho_1,b_1,c,k$)  $n\in \omega $, we have
there is no finite set $Y\subseteq 2^n\cap [\sigma]^\preceq$
with $|Y|\leq k$
such that every $\rho\in 2^n\cap (\rho_1,b_1)$
admits a $c$-solution in $Y$.
\end{claim}
\begin{proof}To see this,
note that otherwise, by compactness
\footnote{More specifically, if $Y\subseteq 2^n\cap [\sigma]$
makes every $\rho\in 2^n\cap (\rho_1,b_1)$ admitting a $c$-solution in $Y$,
then its prefix, namely $Y\uhr m$ (for some $m<n$)
also makes every $\rho\in 2^m\cap (\rho_1,b_1)$ admitting a $c$-solution in $Y\uhr m$.},
there is  a set $Y\subseteq [\sigma]$ of reals
with $|Y|\leq k$ such that
for every $x\in [\rho_1,b_1]$, $x$ admits a
$c$-solution  in $ Y$.
But by paring argument
\footnote{Climbing up along $(\rho_1,b_1)$ and get rid of 
as many elements of $Y$ as possible until 
reaching a $\rho\in (\rho_1,b_1)$ 
and some $y\in Y$ so that inside $(\rho,b_1)$, one cannot get rid of $y$},
there would be a $\rho\in (\rho_1,b_1)$ and $y\in Y$ such that $y\in\Y_c(\rho,b_1)$,
a contradiction to $c$-\maximal\ of $(\rho_1,b_1)$.

\end{proof}

\noindent Now
let $\h c\in\omega$ 
and sufficiently large $n  $
(depending on $\rho_1,b_1,c,\h c,k$)
so that for some
 $\h\rho_0\in (\rho_0,b_0)$
\begin{itemize}
\item $\h\rho_0\in 2^n\cap \D_{\leq \h c}$; moreover,

\item there is no set $Y\subseteq 2^n\cap [\sigma]^\preceq$ with $|Y|\leq O_U(2^{c+\h c})$
such that  every $\rho\in 2^n\cap (\rho_1,b_1)$
admits a $c$-solution in $Y$.

\end{itemize}
By \Cref{numKclaim1} and  \Cref{numKsimultaneousdrop},
such $n,\h c,\h\rho_0$ exist.
Now since $$|\Y_c(\h\rho_0)| \leq  O_U(2^{c+\h c}),$$
 there exists a $\h\rho_1\in (\rho_1,b_1)$
such that $\h\rho_1$ does not admit a $c$-solution in 
$[\sigma]^\preceq\cap \Y_c(\h\rho_0)$. 
Thus, we are done.

\end{proof}

\begin{lemma}\label{numKlem6}
Let $ c\in\omega$ and nonempty $c$-\maximal\ condition
$p=(\rho,b)$.
We have, for every $n\in\omega$,
there is a $\rho^*\in p$
with $|\rho^*|>n$
such that
$$\big|\Y_c(\rho^*) \setminus \big[\ \Y_c(p)\uhr n\ \big]^\preceq\big|
=O_U(2^c).$$ 

\end{lemma}
The point here is that $O_U(2^c)$ is a  constant that does not depend on $\rho^*,b$.
\begin{proof}
Repeatedly applying  \Cref{numKlem5},
we can find $2^b$
many strings
$ \rho_0,\cdots,\rho_{2^b-1}$  in $2^m\cap p$ (for some $m>n$)
such that for every $i\ne j<2^b$,
$$ \big(2^m\setminus \big[\ \Y_c(p)\uhr n\ \big]^\preceq \big)
\cap \Y_c(\{ \rho_i, \rho_j\})=\emptyset.$$
In another word,
\begin{align}\label{numKeq4}
\text{$\Y_c(\rho_i) \setminus \big[\ \Y_c(p)\uhr n\ \big]^\preceq$
are mutually disjoint.}
\end{align}
By \Cref{numKsimultaneousdrop},
there are $\h n\geq m$ and   $\h\rho_i\in (\rho_i,b)$ for each $i<2^b$
such that $\h\rho_i\in 2^{\h n}\cap \D_{\leq b+O_U(1)}$ for all $i<k$.
Therefore,  
\begin{align}\label{numKeq5}
|\cup_{i<2^b} \Y_c(\h\rho_i)|\leq O_U( 2^{b+c})
\end{align}
Thus, by
(\ref{numKeq4})(\ref{numKeq5}), we have: for some $i^*$,
$$
\big|\Y_c(\h\rho_{i^*})\setminus \big[\ \Y_c(p)\uhr n\ \big]^\preceq\big|
= O_U(2^{b+c})/2^b=O_U(2^c).
$$
Thus, $\h\rho_{i^*}$ is the desired $\rho^*$.

\end{proof}

\noindent\textbf{Construction and verification.}

Let sequence of nonempty conditions
$p_0=(\rho_0,b_0)\geq p_1=(\rho_1,b_1)\geq\cdots$
with
$b_s<b_{s+1}$
be such that  for every $c\in\omega$,
every $s\geq c$, we have:
\begin{enumerate}
\item  $p_s$ is $c$-\maximal;

\item 
For every $y\in \Y_c(p_s)$,
we have $\Delta_K(y\uhr n_{c,y})\leq O_U(1)$
for some $|\rho_s|<n_{c,y}<|\rho_{s+1}|$;

\item  
For some number $m_{c,s}$ with
 $\max_{y\in \Y_c(p_s)} n_{c,y}<m_{c,s}<|\rho_{s+1}|$,
some string $\sigma_{c,s}$ with $\rho_s\prec\sigma_{c,s}\prec \rho_{s+1}$ and $|\sigma_{c,s}|\geq m_{c,s}$,
we have 
$$\Y_c(\sigma_{c,s})\setminus [\Y_c(p_s)\uhr m_{c,s}]^\preceq\big|\leq O_U(2^c).$$

\end{enumerate}
The
existence of a nonempty extension  is due to \Cref{numKclaimext};
the three items are made possible due to 
(\ref{numKeqmaximal}), \Cref{numKpropconstantreturnofKtrivial}
and \Cref{numKlem6}.
Let $x:= \cup_s \rho_s$. Clearly $x$ is non i.o. $K$-trivial. 
We show that $x$ is as desired.

Now, pick a $c\in\omega$ and
$y\in \Y_c(x)$,
we have two cases.

\noindent\textbf{Case 1.}
For infinitely many $s\in\omega$, 
$$y\uhr |\sigma_{c,s}|\in [\Y_c(p_s)\uhr m_{c,s}]^\preceq.$$
In this case, by item (2),
there are some $|\rho_s|<n<m_{c,s}$ such that $\Delta_K(y\uhr n)\leq O_U(1)$.
Thus $y$ is i.o. $K$-trivial.

\noindent\textbf{Case 2.}
For all but finitely many $s\in\omega$,
$$y\uhr |\sigma_{c,s}|\in \Y_c(\sigma_{c,s})\setminus [\Y_c(p_s)\uhr m_{c,s}]^\preceq.$$
But due to item (3), the number of such $y$ is bounded by $O_U(2^c)$.

Thus we are done proving \Cref{numKpropfinitenoniosolution}.

\subsection{Intersection of two non i.o. $K$-trivial}

In this subsection, we prove:
\numKpropintersectionoftwo*

The proof is similar to that of \Cref{numKpropfinitenoniosolution}.
We firstly note that given two nonempty conditions
$\t p = (\t\rho,\t b),\h p=(\h\rho,\h b)$,
 we can extend them to nonempty conditions $\t q=(\t\sigma,\t b+1),\h q
 =(\h\sigma,\h b+1)$ 
 such that the intersection $\Y_c(\t\sigma)\cap \Y_c(\h\sigma)$
 is controlled by $\Y_c(\t p)\cap \Y_c(\h p)$.
 For a set $Y$ of reals, we say $Y$ \emph{splits } at level $n$ 
iff for every two distinct $y_0,y_1\in Y$, $y_0\uhr n\ne y_1\uhr n$.
\begin{claim}\label{numKclaim2}
Let $c \in\omega$ and   nonempty $c$-\maximal\ conditions $\t p=(\t\rho,\t b), \h p=(\h\rho,\h b)$.
Let $n\in\omega$
so that $\Y_c(\t p)\cup\Y_c(\h p)$ splits at level $n$.
Then there exist two nonempty $c$-\maximal\ extensions
$(\t\sigma,\t b+1)$ and $(\h\sigma, \h b+1)$ 
of $\t p$ and $\h p$ respectively
with identical stem length greater than $n$ such that 
\begin{align}\label{numKeqclaim2}
\Y_c(\t\sigma)\cap \Y_c(\h\sigma) \subseteq \bigg[\big(\Y_c(\t p)\cap \Y_c(\h p)\big)\uhr n\bigg]^\preceq;
\end{align}
moreover, 
$$
\Y_c(\t\sigma,\t b+1)\cap \Y_c(\h\sigma,\h b+1) = \Y_c(\t p)\cap \Y_c(\h p).
$$
\end{claim}
\begin{proof}
 The first part (\ref{numKeqclaim2}) follows by repeatedly applying \Cref{numKlem5}
 (with $p_0,p_1$ substituted by $\t p,\h p$)
 with respect to every $\sigma\in 2^n\setminus \big(\Y_c(\h p)\uhr n\big)$;
 and then repeatedly applying \Cref{numKlem5} again
(with $p_0,p_1$ substituted by $\h p,\t p$)
with respect to every $\sigma\in 2^n\setminus  \big(\Y_c(\t p)\uhr n\big)$.
In the end, we get $\t\tau\in \h p, \h\tau\in \h p$
such that 
\begin{align}\label{numKeqclaim2-1}
\Y_c(\t\tau)\cap \Y_c(\h\tau) \subseteq \bigg[\big(\Y_c(\t p)\cap \Y_c(\h p)\big)\uhr n\bigg]^\preceq.
\end{align}
Note that (\ref{numKeqclaim2-1}) will be preserved for any future extension of $(\t\tau, \t b), (\h\tau,\h b)$.

To address the ``moreover" part,
we need to extend $\t\tau,\h\tau$ one by one to 
prevent new common solution of each other to become intersection
of common solution.
First, we may assume
\begin{align}\label{numKeq6}
\bigg(\Y_c(\h p)\setminus \Y_c(\t p)\bigg)\cap [\Y_c(\t\tau)]=
 \emptyset
\end{align}
otherwise extend $\t\tau$ to kill each $y\in \Y_c(\h p)\setminus \Y_c(\t p)$.
Now use \Cref{numKclaimext} to extend
$(\t\tau, \t b) $ to nonempty
$c$-\maximal\ condition
$(\t\eta,\t b+1)$
(don't do this for the other condition for now)
and then extend $(\h\tau, \h b)$
to kill each  $y\in \Y_c(\t\eta,\t b+1)\setminus \Y_c(\h p)$.
i.e., Extend $(\h\tau,\h b)$ to some nonempty condition  $(\h\eta,\h b)$
so that 
\begin{align}\label{numKeq7}
\bigg(\Y_c(\t\eta,\t b+1)\setminus \Y_c(\h p)\bigg)\cap [\Y_c(\h\eta)] = \emptyset.
\end{align}
Finally, by extension of condition 
\Cref{numKclaimext}, extend $(\t\eta,\t b+1)$
and $(\h\eta,\h b)$ to nonempty $c$-\maximal\ condition
$(\t\sigma,\t b+1)$ and $(\h\sigma,\h b+1)$ with $|\t\sigma| = |\h\sigma|$.
Now we have
(since $\Y_c(\h\sigma,\h b+1)\subseteq [\Y_c(\h\eta)] $ and since
$\Y_c(\t\sigma,\t b+1)\subseteq [\Y_c(\t\tau)]$):
\begin{align}
&\Y_c(\t\sigma,\t b+1)\cap \Y_c(\h\sigma,\h b+1)\\ \nonumber
  \subseteq& \Y_c(\t\sigma,\t b+1)\cap [\Y_c(\h\eta)]\cap \Y_c(\h\sigma,\h b+1)\cap [\Y_c(\t\tau)] \\ \nonumber
\text{by (\ref{numKeq7})}\ \ \ 
\ \ \subseteq & \Y_c(\h p)\cap \Y_c(\h\sigma,\h b+1)\cap [\Y_c(\t\tau)]\\ \nonumber
 \subseteq & \Y_c(\h p)\cap [\Y_c(\t\tau)] \\ \nonumber
\text{by (\ref{numKeq6})} \ \ \ \ \ 
\subseteq& \Y_c(\h p)\cap \Y_c(\t p).
\end{align} 
Thus, we are all done.

\end{proof} 

We strengthen \Cref{numKclaim2} to deal with multiple $(\Y_c:c\leq c_0)$
simultaneously.
\begin{claim}\label{numKclaim22}
Let $c_0 \in\omega$ and  let
 nonempty conditions $\t p=(\t\rho,\t b), \h p=(\h\rho,\h b)$
 be $c$-\maximal\ for all $c\leq c_0$. 
Let $n\in\omega$
so that $\Y_c(\t p)\cup\Y_c(\h p)$ splits at level $n$ for all $c\leq c_0$.
Then there exist two nonempty  extensions
$\t q=(\t\sigma,\t b+1)$ and $\h q=(\h\sigma, \h b+1)$ 
of $\t p$ and $\h p$ respectively
with identical stem length greater than $n$ 
and with $\t q,\h q$ being $c$-\maximal\ for all $c\leq c_0$  such that 
for every $c\leq c_0$
\begin{align}\label{numKeqclaim22}
\Y_c(\t\sigma)\cap \Y_c(\h\sigma) \subseteq \bigg[\big(\Y_c(\t p)\cap \Y_c(\h p)\big)\uhr n\bigg]^\preceq;
\end{align}
moreover, 
$$
\Y_c(\t\sigma,\t b+1)\cap \Y_c(\h\sigma,\h b+1) = \Y_c(\t p)\cap \Y_c(\h p).
$$
\end{claim}
\begin{proof}
The same as  \Cref{numKclaim2}.
\end{proof}

\begin{proof}[Proof of \Cref{numKpropintersectionoftwo}]
Using \Cref{numKclaim22} (repeatedly),
we obtain
 two sequence of nonempty conditions  
$ \t p_0=(\t\rho_0,\t b_0)\geq \t p_1=(\t\rho_1,\t b_1)\geq\cdots$
and $ \h p_0=(\h\rho_0,\h b_0)\geq \h p_1=(\h\rho_1,\h b_1)\geq\cdots$
with $\t b_s<\t b_{s+1}, \h b_s<\h b_{s+1}$ and with 
$|\t\rho_s|=|\h\rho_s| $
such that for every $c\in\omega$,  
\begin{enumerate}
\item $\t p_s,\h p_s$ are $c$-\maximal\ for all $s\geq c$
and $\Y_c(\t p_c)\cup \Y_c(\h p_c)$ splits at level 
$|\t\rho_c|$;

\item For every $s>c$, there exists $|\t\rho_{s-1}|<n_{c,s}<|\t\rho_s|$ such that
$$\Y_c(\t\rho_s)\cap \Y_c(\h\rho_s)\subseteq \bigg[\big(\Y_c(\t p_c)\cap \Y_c(\h p_c)\big)\uhr n_{c,s}\bigg]^\preceq.$$

\end{enumerate}
Thus, let $\t x:= \cup_s \t\rho_s, \h x:= \cup_s \h\rho_s$.
Clearly $\t x,\h x$ are non i.o. $K$-trivial.
On the other hand, item (2) entails that $$\Y_c(\t x)\cap \Y_c(\h x)=\Y_c(\t p_c)\cap \Y_c(\h p_c)$$ 
for all $c\in\omega$.
Thus,  by \Cref{numKlemcommonsolutionisKtrivial}
 $\Y_c(\t x)\cap \Y_c(\h x)$ is a set of $K$-trivial.
 
 Thus, we are done.

\end{proof}

%% file: arbitraryslowgrowth.tex
\section{Arbitrary slow growth and uncountable solutions}
\label{numKsecslowgrowth}

 The goal of this section is to prove \Cref{numKpropslowgrowth}.
 \numKpropslowgrowth*
 
 From the view of Section \ref{numKsecsmall},
 this boils down to extending a nonempty condition
 $p$ to alow $\Y_c(p)$ to grow more branches above an 
 arbitrary high level while
 while keeping the new stem's $\Delta_K$ complexity low enough.
 We starts by establish some properties of nonempty conditions.
 For a real $z\in 2^\omega$, we write $\Y_c^{-1}(z)$ for the 
 set of real $x$ so that $z\in \Y_c(x)$.

\begin{claim}\label{numK-claim-positivemeasurecertainclass}
Let closed set $Q\subseteq 2^\omega$ 
so that $Q$ contains a 1-random real $x_0$
and suppose for some $b\in\omega$, for every $y\in 2^\omega\setminus Q$,
there exists a prefix $\tau$ of $y$ 
such that $ \tau\in \D_{\leq b}$ and $[\tau]\cap Q=\emptyset$.
Then we have $\m(Q)>0$.

\end{claim}
\begin{proof}
 
Pick a sufficiently long (depending on $b$) prefix $\sigma$ of $x_0$
(so $K(\sigma)>|\sigma|/2$
is sufficiently large depending on  $ b$).
By \Cref{numKclaimext} item (1),
 $ [\sigma, \h b]$ has positive measure
 where $\h b>b$ is  sufficiently large depending on $b$.
 But since $2^\omega\setminus Q\subseteq [\D_{\leq b}]$
 (and since $[\sigma]\cap Q\ne\emptyset$),
 we have $[\sigma,b]\subseteq Q$.
 Thus we are done.

\end{proof}

 \begin{restatable}{claim}{numKclaimpositiveconditiona}
 \label{numKclaimpositiveconditiona}
 Let $c\in\omega$ and let nonempty condition $p=(\rho,b)$
 with $b$ sufficiently large depending on $U$.
 Suppose $\Y_c(p)$ contains some eventually-zero real $y$.
 Then for every $N$, there exists an eventually-zero real $z\ne y$
 with $z\uhr N=y\uhr N$ such that 
 $$
 \m([p]\cap \Y_c^{-1}(z))>0.
 $$
 
 \end{restatable}
 \begin{proof}
By  \Cref{numKclaimext}, $\m([p])>0$, so there is a 
1-random real $x_0\in [p]$.
We note that for every sufficiently long
(depending on $y$ and $|\sigma|>N$) $\sigma\prec x_0 $,
say $|\sigma| = m$, let $y_m$ be $y$ with the $m^{th} $ bit reversed
(i.e., $y_m(n)=y(n)$ for all $n\ne m$ and $y_m(m) = 1-y(m)$)
we have 
$y_m\leq_K^c x_0$.
To see this,  note that for every $n\geq m$
(since $m$ is sufficiently large),
\begin{align}\nonumber
K(y_m\uhr n)&\leq K(y\uhr n)+K(m)+O_U(1)\\ \nonumber
&\leq K(n)+K(m)+O_y(1)<n/2\leq K(x_0\uhr n);
\end{align}
while for every $n<m$, $$K(y_m\uhr n)\leq K(y\uhr n)\leq K(x_0\uhr n)+c.$$

Now consider  the class:
$$
[p]\cap \Y_c^{-1}(y_m).
$$
Note that 1-random real $x_0$ is in this class
and this class satisfies the hypothesis of 
\Cref{numK-claim-positivemeasurecertainclass}
(since $y_m$ is $K$-trivial),
therefore (by 
\Cref{numK-claim-positivemeasurecertainclass}), $$\m([p]\cap \Y_c^{-1}(y_m))>0.$$
Thus we are done.
\end{proof}

The following \Cref{numKclaimpositiveconditionb}
shows that we can always
shrink to $\Y_c^{-1}(z)$ for some new eventually-zero 
real $z$ without losing too much measure of a condition
while finding a new stem
whose $\Delta_K$ complexity low enough
in the shrinked area.
 
\begin{restatable}{lemma}{numKclaimpositiveconditionb}
\label{numKclaimpositiveconditionb}
Let $c\in\omega$ and $b\in\omega$ be sufficiently large
(depending on $U$).
Let nonempty condition $p= (\rho,b)$
and eventually-zero real $z\in 2^\omega$ so that 
$$
\m([p]\cap \Y_c^{-1}(z))>0.
$$
Then, for every $N\in\omega$, there exists a nonempty extension $q= (\h\rho,  b)$
of $p$ with $|\h\rho|>N$ such that
$\Delta_K(\h\rho)\leq 3b$ and 
$$
\m([q]\cap \Y_c^{-1}(z))>0.
$$

\end{restatable}
The proof is delayed to Section \ref{numKsecpositive}.
We now show how to construct the next condition,
which mainly combines \Cref{numKclaimpositiveconditionb} with
the slow growth of $\Delta_K$-complexity entailed by function $h$.
\begin{lemma}\label{numKlemslowgrow}
Let $c\in \omega$
and let nonempty condition 
$p =(\rho ,b )$ with $b $ sufficiently large (depending on $U$)
and $\rho $ sufficiently long
 so that
 \begin{align}\label{numKeq9}
\text{  $h(|\rho |)$ is sufficiently large depending on $b $.}
  \end{align}
  Suppose $\Y_c(p)$ contains an eventually-zero real $y$.
  We have 
  for every $N\in\omega$,  there exists 
  an eventually-zero real $z\ne y$ with $z\uhr N=y\uhr N$,
  a nonempty extension $q=(\h\rho,\h b)$ of $p$ with $b<\h b$
  such that 
  $z\in \Y_c(q)$ and for every $n\in [|\rho|,|\h\rho|]$,
  $$
  \inf_{m\geq n}\Delta_K(\h\rho\uhr m) \leq h(n).
  $$
  Moreover, we can chose $\h\rho$ so that $|\h\rho|$ is sufficiently long so that 
   \begin{align}\nonumber
\text{  $h(|\h\rho |)$ is sufficiently large depending on $\h b $.}
  \end{align}
  
\end{lemma}
\begin{proof}
By \Cref{numKclaimpositiveconditiona},
let 
 eventually-zero real $z\ne y$
with $z\uhr N=y\uhr N$
such that 
$$
\m([p]\cap \Y_c^{-1}(z))>0.
$$
Let 
$a_z$ be the $K$-trivial constant of $z$, namely $\Delta_K(z\uhr n)\leq a_z $ for all $n$.

By \Cref{numKclaimpositiveconditionb},
there exits
 a string $\sigma \in p $ with $\sigma $ sufficiently long
 so that
 \begin{align}\label{numKeq8}
 h(|\sigma |)> 9b \vee 3a_z 
 \end{align}
  such that
  \begin{align}\nonumber
  \text{ $\Delta_K(\sigma )\leq 3b $
 and } \m([\sigma ,b ]\cap \Y_c^{-1}(z))>0.
 \end{align}

Since the set of non 1-random reals has measure zero,
there exists a 1-random $x_0$ in $ [\sigma ,b ]\cap \Y_c^{-1}(z)$.
Thus,
let $\h b:=  3b \vee a_z$ there must be a prefix
$\sigma\preceq \h\rho$ of $x_0$ such that 
$q=(\h\rho,\h b)$ is a nonempty condition
and $z\in \Y_c(q)$.
We may assume,
by \Cref{numKclaimext} and \Cref{numKsimultaneousdrop}, $\Delta_K(\h\rho) \leq 3\h b$
and 
\begin{align}\nonumber
h(|\h\rho|)\text{ is sufficiently large depending on $\h b$}.
\end{align}
Since otherwise we can
(by \Cref{numKclaimext} and \Cref{numKsimultaneousdrop}) extend $\h\rho$ to achieve these.
We note that for every $n\in [|\rho|,|\sigma|]$,
$$
\inf_{m\geq n}\Delta_K(\h\rho \uhr m)\leq
\Delta_K(\sigma )\leq 3b \overset{(\ref{numKeq9})}{\leq} h(|\rho |)\leq   h(n);
$$
 On the other hand, for every $n\in [|\sigma |,|\h\rho|]$,
$$
\inf_{m\geq n}\Delta_K(\h\rho \uhr m)\leq
\Delta_K(\h\rho )\leq 3\h b \overset{(\ref{numKeq8})}{\leq} h(|\sigma|)\leq   h(n).
$$
Thus we are done.
\end{proof}

Now it is clear how to prove \Cref{numKpropslowgrowth}.
\begin{proof}[Proof of  \Cref{numKpropslowgrowth}]
Let $c\in\omega$ be sufficiently large depending on $U$ and 
starting with a nonempty condition $p_0=(\rho_0,b_0)$
where $b_0$ is  sufficiently large depending on $U$ 
so that $0^\omega\in \Y_c(p_0)$
and $ \rho_0 $ is sufficiently long so that $h(|\rho_0|)$
is sufficiently large depending on $b_0$.
Repeatedly applying \Cref{numKlemslowgrow}, we obtain
a sequence of conditions
$p_0\geq p_1=(\rho_1,b_1)\geq p_2=(\rho_2,b_2)\geq\cdots$
with $b_s<b_{s+1}$ such that  for every $s\in\omega$
\begin{itemize}
\item  $\Y_c(p_s)$ contains a finite perfect tree set\footnote{
We mean there is an embedding from $2^n$ to the set of prefix of $\Y_c(p_s)$.}
and every leaf of that perfect tree splits in 
the perfect tree set of $\Y_c(p_{s+1})$.

\item For every $n\in [|\rho_s|,|\rho_{s+1}|]$,
 $\inf_{m\geq n}\Delta_K(\rho_{s+1}\uhr m)\leq h(n).$ 

\end{itemize}
Thus, let $x:= \cup_s \rho_s$. We have $x$ is 
the desired real for \Cref{numKpropslowgrowth}.
Thus we are done proving  \Cref{numKpropslowgrowth}. 
\end{proof}

\begin{remark}
\label{numKrem2}

Combining the argument of \Cref{numKpropfinitenoniosolution}
and \Cref{numKpropslowgrowth},
we note that the real $x$ in \Cref{numKpropfinitenoniosolution} can be 
chosen so that it has uncountably many solutions.
\end{remark}

%% file: constantdropofKtrivial.tex
\section{Ample excess lemma for non i.o. $K$-trivial}
\label{numKsecampleexcess}
In this section, we prove the ample excess lemma for non i.o. $K$-trivial.
\numKampleexcess*

\Cref{numKampleexcess} is mainly reduced to the following Lemma.

\begin{lemma}\label{numKlemlowerbound}
For every $\sigma\in 2^{<\omega}$,
suppose  $|\sigma|\in [2^n,2^{n+1})$.
If $$K(\sigma)\leq n+K(n)+O_U(1),$$
then $\Delta_K(\sigma\uhr m) \leq O_U(1)$
for some $m\geq 2^{n-1}$.
\end{lemma}
\begin{proof}

Recall that a \emph{Kraft function}   is a right c.e. function
$\h K$ such that $$\sum_{\sigma\in 2^{<\omega}} 2^{-K(\sigma)}<\infty.$$
Consider such a Kraft function $\h K$, 
roughly speaking, its goal is:
for every $n\in\omega$ and binary string
$\sigma $ of length $\in [2^n,2^{n+1})$,

\begin{align}\label{numKeqtarget3}
&\text{if $K(\sigma)\leq n+K(n)+O_U(1)$, then $\h K$ makes }\\ \nonumber
&\text{$\h K(\sigma\uhr m)\leq n+K(n)$ for some $m\in [2^{n-1},2^n)$ }\\ \nonumber
&\text{so  that $K(m)\geq n+K(n)-O_U(1)$.}
\end{align}
If (\ref{numKeqtarget3}) can be achieved and if $\h K$ is a Kraft function,
then, 
\begin{align}\nonumber
K(\sigma\uhr m)&\leq \h K(\sigma)+O_U(1) \\ \nonumber
& \leq n+K(n)+O_U(1)\leq K(m)+O_U(1).
\end{align}
The point is, for over $2/3$ proportion
  of integers $m\in [2^{n-1},2^{n+1})$, we have
\begin{align}\label{numKeq81}
K(m)\geq n+K(n)-O_U(1).
\end{align}
This is seen by regarding an integer $m\in [2^n,2^{n+1})$
as its binary expansion string.
Thus, more than $1/3$ proportion of integers in $[2^{n-1},2^n)$
satisfies (\ref{numKeq81})
\footnote{For convenience, we assume this is true online. i.e., 
For every round $t\in\omega$, more than $1/3$ proportion of integers $m$ in $[2^{n-1},2^n)$
satisfies $K[t](m)\geq n+K[t](n)-O_U(1)$.}.
The main remaining issue
(to realize (\ref{numKeqtarget3})) is to find $m$ in an online fashion;
and verifying that the cost of $\h K$ is finite.

\noindent\textbf{Definition of $\h K$.}
  We say binary string $\sigma$
    (say of length $\in [2^n,2^{n+1})$) is \emph{active} (at round $t$)
 iff
\begin{align}\nonumber
K[t](\sigma) \leq n+ K[t](n)+O_U(1).
\end{align}
For each $n\in\omega$, $\h K$ works
for $n$ under the working assumption  that 
\begin{align}
\text{$K[t](n)$ will no longer change.}
 \end{align}
  If this assumption 
 became false in the future, then player $\h K$ 
  restart the work for $n$.
 A period for $n$ is an interval $[t_0,t_1)$ so that 
 $K[t](n)$ does not change within this period but changed
 at the round $t_0, t_1$.
 Let $M_n[t]$ denote the set of number $m\in [2^{n-1}, 2^n)$
 so that $$K[t](m)\geq n+K[t](n)-O_U(1).$$
 
 \textbf{Working for $n$ in a period $[t_0,t_1)$ (of $n$):}.
 \begin{enumerate}
 \item  At each round $t$,
 for each newly added 
 \footnote{Meaning $\sigma$ is not active at round $t-1$ but become active at round $t$. 
 At beginning round, all active strings are considered newly added.}
  active binary string $\sigma$ with $|\sigma|\in [2^n,2^{n+1})$,
assign a $m_\sigma $ in $M_n[t] $
 such that
 \begin{itemize}
 \item each $m\in [2^{n-1},2^n)$ is  assigned (in this period) to at most $O_U(1)$ many binary strings 
 \footnote{This is possible since there are
 $\leq O_U(2^n)$ many strings   with length $\in [2^n,2^{n+1})$
 that admits $U$-description shorter than $n+K(n)+O_U(1)$.};
 \item then player $\h K$ acts by letting  
 $$\h K[t ](\sigma\uhr m_\sigma) = n+K[t_0](n).$$
 \end{itemize}
 
 \item If (during this period), an $m\in [2^{n-1},2^n)$ is kicked out of $ M_n[t]$,
 then for each $\sigma$ assigned to $m$,
 assign  $\h m\in M_n[t]$ to them and then
 $\h K$ re-assign mass to $\sigma\uhr \h m$.
 \end{enumerate}

\noindent \textbf{Cost of $\h K$.}
Since during each period of $n\in\omega$,
there are at most $O_U(2^n)$ many binary strings with length $\in[2^n,2^{n+1})$
becomes active,
  the total cost of $\h K$ to address 
  newly added active strings is:
  \begin{align}\nonumber
  \leq&\sum_{n\in\omega}\ \ \ \sum\limits_{\substack{\text{period }\\ 
 \text{$[t_0,t_1)$ of $n$}}}\ \ \  O_U(2^n)\cdot 2^{-n-K[t_0](n)}\\ \nonumber
  \leq& O_U(1)\cdot\sum_{n\in\omega}\ \ \ \sum_{\substack{t\in\omega\\  K[t](n)<K[t-1](n)}}
   2^{-K[t](n)}\\ \nonumber
 \leq& O_U(1).
 \end{align}
 
 For each  $n\in\omega$,  at each round $t$,
the cost for re-assign mass due to 
$m$ being kicked out of $M_n$ (at this round)
is bounded by $$
O_U(2^{-n-K[t](n)}) 
$$
since (within each period) each $m$ is only assigned to $O_U(1)$ many binary strings.
While the cost for $K$ to kick $m$ out of $M_n$ is greater than
$$ 2^{-n-K[t](n)}.$$ 
 Thus the cost of $\h K$ to re-assign mass is bounded by  $O_U(1)$ times the total cost of $K$.
  In summary, we have that the cost of $\h K$ is bounded by $O_U(1)$.
  Thus we are done proving \Cref{numKlemlowerbound}.

\end{proof}

\begin{proof}[Proof of \Cref{numKampleexcess}]
Since $x$ is not i.o. $b$-$K$-trivial,
by \Cref{numKlemlowerbound}, we have
\begin{align}\label{numKeq82}
K(x\uhr n)\geq \log_2n+ K([\log_2 n]) -O_{U,x}(1)\text{ \ \ \ for all $n\in\omega$}.
\end{align}

On the other hand, it is easy to see that 
for every $\rho\in 2^{<\omega}$
with $|\rho| = 2^n$,
$$
\sum_{\sigma\in [\rho]^\preceq\cap 2^{<2^{n+1}}} 2^{-K(\sigma)}\leq O_U(2^{-K(\rho)}).
$$
In particular,
\begin{align}\label{numKeq80}
\sum_{m\in [2^n,2^{n+1})} 2^{-K(x\uhr m)} \leq O_U(2^{-K(x\uhr 2^n)}) .
\end{align}

Thus,
\begin{align}\nonumber
\sum_{m\in\omega}2^{-(K(x\uhr m)-\log_2 m)}
&= \sum_{n\in\omega}\sum_{m\in [2^n,2^{n+1})} 2^{\log_2 m-K(x\uhr m)}\\ \nonumber
&\overset{(\ref{numKeq80})}{\leq}  O_U(1)\cdot\sum_{n\in\omega}2^{n-K(x\uhr 2^n)}\\ \nonumber
  &\overset{(\ref{numKeq82})}{\leq}  O_{U,x}(1)\cdot\sum_{n\in\omega}
2^{n-(\log_2 2^n)-K([\log_2 2^n])}\\ \nonumber
&=O_{U,x}(1)\cdot\sum_{n\in\omega}2^{-K(n)}<\infty.
\end{align}

\end{proof}

We point out a complexity drop phenomenon for $K$-trivial.
\begin{proposition}
\label{numKpropconstantreturn2}
For every real $x$, if $K(x\uhr 2^n)\leq n+K(n)+O_U(1)$ for infinitely many $n\in\omega$, 
then  
$\Delta_K(x\uhr n)\leq O_U(1)$ for infinitely many $n\in\omega$.
In particular, if $x$ is $K$-trivial,
then $\Delta_K(x\uhr n)\leq O_U(1)$ for infinitely many $n\in\omega$.
\end{proposition}
\begin{proof}
The first part follows directly from \Cref{numKlemlowerbound}.
For the ``In particular" part,
since $x$ is $K$-trivial, we obviously have
$$
K(x\uhr 2^n)\leq K(2^n)+O_{U,x}(1)\leq K(n)+O_{U,x}(1).
$$
Thus, the conclusion follows from the first part.
\end{proof}

\section{Proof of technical results.}
\label{numKtechnicalresults}

\subsection{Positive measure of condition}
\label{numKsecpositive}
In this section,
we estimate the measure of certain closed set
such as a condition
and prove \Cref{numKclaimext} and \Cref{numKclaimpositiveconditionb}.

\existenceofextension*
\begin{proof}
For item (1):
Let $\sigma$ be the shortest string 
in $[\rho]^\preceq\cap \D_{<b}$.
Since $$|2^n\cap \D_{< b}|\leq O_U(2^b),$$ we have
\begin{align}\label{numKeq40}
\m([p]|\rho)&= 1-\m\bigg( [\rho]^\preceq\cap \D_{<b} \bigg|\rho \bigg)\\ \nonumber
&\geq 1-O_U(2^{b-(|\sigma|-|\rho|)}).
\end{align}
On the other hand,
\begin{align}\nonumber
K(\rho)&\leq K(\sigma)+K(|\sigma|-|\rho|)+O_U(1)\\ \nonumber
\text{since $\sigma\in \D_{<b}$\ \ \ \ \ }
&< K(| \sigma|)+  b+ 2\log_2 (|\sigma|-|\rho|)+O_U(1)\\ \nonumber
&\leq K(|\rho|)+ K(|\sigma|-|\rho|)+b+2\log_2 (|\sigma|-|\rho|)+O_U(1)\\ \nonumber
&\leq K(|\rho|)+ b+ 4\log_2  (|\sigma|-|\rho|)+O_U(1)\\ \nonumber
&\leq K(|\rho|)+ b+ \log_2 b+ 4\log_2  (|\sigma|-|\rho|).
\end{align}
The last inequality is due to $b$ sufficiently large depending on $U$.
In summary,
$$
|\sigma|-|\rho|\geq \Delta_K(\rho)-b-\log_2 b +O(1).
$$
Combine with (\ref{numKeq40}),
we have
\begin{align}
\m([p]|\rho)&\geq  1-O_U(2^{b-\Delta_K(\rho)+b+\log_2 b })\\ \nonumber
&\geq 1-2^{-(\Delta_K(\rho)-2b-2\log_2 b)}.
\end{align}
The last inequality is due to    $b$ sufficiently large depending on $U$.

For item (2): Suppose otherwise $[p]$ has measure zero. 
By item (1), for every $x\in [p]$, $\Delta_K(x\uhr n)\leq O_p(1)$,
so $x$ is a $K$-trivial. But let $x\in [p]$, by \Cref{numKpropconstantreturn2},
there is some $n\in \omega$ such that $\Delta_K(x\uhr n)\leq O_U(1)<b$
(since $b$ is sufficiently large depending on $U$).
Thus, $[p]=\emptyset$, a contradiction.

Item (3) follows from item (1) directly. 
Let 1-random real $x_0\in [p]$
(which must exists since $[p]$ has positive measure).
Pick a sufficiently long (depending on $p,n ,\h b$ ) prefix $\sigma$ of $x_0$
so that $\Delta_K(\sigma)>3\h b$,
by item (1), we have $(\sigma,\h b)$ is nonempty.

\end{proof}


\numKclaimpositiveconditionb*
\begin{proof}

First we note that for every eventually-zero real $\h z$,
there are infinitely many integer intervals $I$ with each of them sufficiently long (to be specified)
such that for every level $n$ in these intervals,
$\Delta_K(\h z\uhr n)\leq O_U(1)$.
This is simply because, we can computably enumerate all eventually-zero reals; so 
let $n_0<n_1<\cdots$ so that each interval $[n_i, n_{i+1}]$ is sufficiently long
\footnote{It is sufficiently long depending ont the $z $ it is assigned to.},
 computably  assign each interval an eventually-zero real so that each eventually-zero real is assigned 
to infinitely many intervals. Clearly, we have for every $n\in\omega$ in an interval
assigned to $\h z$, $\Delta_K(\h z\uhr n)\leq O_U(1)$.
For convenience, we also assume that for every $I$ assigned to $z$,
\begin{align}\label{numKeqlevel}
\text{$z$ is all zero above level $\min I$.
}
\end{align}
Now let $x_0\in [p]\cap \Y_c^{-1}(z)$ be 1-random
(which exists since $\m([p]\cap \Y_c^{-1}(z))>0$).
Let $\sigma\prec x_0$ so that $|\sigma|>N$, $K(\sigma)>|\sigma|/2$
and $|\sigma|$ is the starting point of some interval $I$ assigned to $z$.
We will extend $\sigma$ to let the $\Delta_K$-complexity drops below certain threshold
and maintaining that certain class (as in the conclusion of \Cref{numKclaimpositiveconditionb})
 has positive measure
 (the later will be achieved due to $I$ long enough so that for some complex  real 
 extending $\sigma$, when it leaves the interval the $\Delta_K$-complexity 
 has return back to some larger threshold beyond that of $z$).
To this end, consider, for each $m\in\omega$, the string $\sigma  0^m$
and condition $p_m:= (\sigma 0^m,b)$.
We note that
\begin{claim}
For some $m\in I-|\sigma|$ with $m$ much smaller than $|I|$, 
the $\Delta_K$-complexity of $\sigma 0^m$ drops below $O_U(1)$
\end{claim}
\begin{proof}
\Cref{numKpropconstantreturn2} not only prove the existence of such $m$
but gives a computable upper bound for $m$.
Note that
\begin{align}\label{numKeq60}
 \text{If }|\sigma 0^m| &= 2^n\text{ and $n$ is sufficiently large 
depending on $|\sigma|$, then}
\\ \nonumber
 K(\sigma 0^m)&\leq K(\sigma)+K(|\sigma 0^m|)+O_U(1)\\ \nonumber
&   \leq 2|\sigma|+K(2^n)
 \\ \nonumber
 &  \leq 3|\sigma|+ K(n)\leq 4|\sigma|+ 2\log_2 n.
 \end{align}
 So by \Cref{numKpropconstantreturn2}, as long as $4|\sigma|+2\log_2 n<n $,
 we have $$ \Delta_K(\sigma 0^m\uhr k)\leq O_U(1)$$
 for some $k\in\omega$.
 The proof of \Cref{numKpropconstantreturn2} shows that
 we can choose witness $k$ so that each chosen witness is shared by at most $O_U(1)$ many 
 $m$.
Therefore, since $I$ is sufficiently long (so there are many $m\in I-|\sigma|$ satisfying 
 (\ref{numKeq60})), we have for  of   $m$ as  (\ref{numKeq60}), the witness $k$ satisfies $k>|\sigma|$.
Thus,
since $I$ is sufficiently long,  we can choose $m$ much smaller (depending on $z$) than $\max I$.
Clearly we can further entail that $m$ is much smaller  (depending on $z$) than $|I|$
by letting $I$ to be sufficiently long.
\end{proof}
In particular $\Delta_K(\sigma 0^m)\leq 3b$ for some $m\in I-|\sigma|$.
Let $m_0\in I-|\sigma|$ be the smallest integer so that  $$\Delta_K(\sigma 0^{m_0})\leq  3b.$$
 If $ [p_{m_0}]\cap \Y_c^{-1}(z) $ has positive measure,
 then we are done.
 Otherwise,
 let $m_1\leq m_0$  be the largest number in $I-|\sigma|$ such that 
 $ [p_{m_1}]\cap \Y_c^{-1}(z) $ has positive measure.
  This $m_1$ exists since  $ [p_0]\cap \Y_c^{-1}(z) $ has positive measure
  and $p_{m+1}\leq p_m$ for each $m\leq m_0$
  (recall that $p_0$ is $(\sigma, b)$
  and $[p_0]\cap \Y_c^{-1}(z)$
  contains the 1-random real $x_0$ therefore has positive measure
  by \Cref{numK-claim-positivemeasurecertainclass}).
It suffices to show that $\Delta_K(\sigma 0^{m_1})\leq 3b$.

  Let's carefully look at what makes 
  $[p_{m_1+1}]\cap \Y_c^{-1}(z)$ to have zero measure.
  Let $a_z$ be a $K$-trivial constant for $z$,
i.e., $\Delta_K(z\uhr n)\leq a_z$ for all $n\in\omega$.
Since (recall (\ref{numKeqlevel}))
$z$ is all zero above level $\min I$, we can choose $a_z$ so that
$$a_z\leq K(z\uhr \min I)\leq 2\min I.$$
Also recall that $m_0$, therefore $m_1$ is much smaller than $|I|$.
Since $I$ is assigned to $z$, so 
$z\uhr m $ is a $c$-solution  to all strings in $  2^m\cap [\sigma]^\preceq$ for all $m\leq \max I$.
Therefore, the complement of the closed set $[p_{m_1+1}]\cap \Y_c^{-1}(z)$
(inside $[\sigma0^{m_1+1}]$)
is contained in
(the union of) the following two open sets:
$$
[2^{\geq \max I}\cap \D_{\leq a_z}], \bigg[[\sigma 0^{m_1+1}]^\preceq\cap \Delta_{<b}\bigg] 
$$
Since $2^n\cap |\D_{\leq a}| = O_U(2^a)$,
we have 
\begin{align}
\m\bigg([2^{\geq \max I}\cap \D_{\leq a_z}]\bigg) \leq O(2^{a_z-\max I})
\leq O(2^{-|I|/2})
<\frac{1}{2}\m(\sigma0^{m_1+1}).
\end{align}
The last inequality is due to $I$ sufficiently long depending on $z$
(recall that $a_z\leq 2\min I$) and $m_1 $ much smaller (depending on $z$) than $|I|$.
Therefore 
$$\m\bigg([\sigma 0^{m_1+1}]^\preceq\cap \Delta_{<b} \bigg|\sigma 0^{m_1+1} \bigg)\geq \frac{1}{2} $$
so $$\m\bigg([p_{m_1+1}]\bigg|\sigma 0^{m_1+1}\bigg)<\frac{1}{2}.$$
By \Cref{numKclaimext} item (1),
$$\Delta_K(\sigma 0^{m_1+1})-2b-2\log_2 b\leq 1,$$
which obviously implies $$\Delta_K(\sigma 0^{m_1})\leq \Delta_K(\sigma 0^{m_1+1})+O_U(1)\leq 3b.$$
Thus we are done.

\end{proof}

\input{Kcomplexityofcommon}

%% file: Kcomplexityofcommon.tex
\subsection{$K$-complexity of common solution and condition}
\label{numKseccomplexityofcommonsolution}
In this subsection,
we prove \Cref{numKlemcommonsolutionisKtrivial} and \Cref{numKsimultaneousdrop}.
\numKlemcommonsolutionisKtrivial*
\begin{proof}
Suppose $p = (\rho,b)$
and without loss of generality assume $ \rho $ is sufficiently long
(depending on $U,b$) so that, in particular, $\m([\rho, 3b]|\rho)>1/2$  
(since otherwise we can extend $\rho$ to achieve this by \Cref{numKclaimext} item (2)).

 We firstly recall some results of Bienvenu, Downey, Nies and Merkle\cite{BDNM15} characterizing K-trival
 as those whose K-complexity is dominated by a Solovay function.
 A computable function $h\in\omega^\omega$ is a \emph{Solovay function}
 iff  $h$ dominates $K$ up to an additive constnat, i.e.,  $$K(n)\leq h(n)+O_h(1)$$ for all $n\in\omega;$ 
 and for infinitely many $n\in\omega$, $$h(n)\leq K(n)+O_h(1).$$
By \cite{BDNM15} Theorem 4.4, 
\begin{theorem}[\cite{BDNM15}]
\label{numKth0}
For every Solovay function $h\in\omega^\omega$,
 every real $x\in 2^\omega$, if $K(x\uhr n)\leq h(n)+O_{x,h}(1)$,
 then $x$ is K-trivial. The converse is also true.
\end{theorem}
Therefore,
if suffices to show that for some Solovay function $h$, for all $n\in\omega$,
\begin{align}\label{numKeq53}
\min_{\rho\in 2^n\cap p} K(\rho)\leq h(n)+O_{p,h}(1).
\end{align}

Recall \cite{BDNM15} Corollary 2.4
\begin{proposition}
[\cite{BDNM15}]
\label{numKpropsolovay}
There is a nondecreasing Solovay function $h\in\omega^\omega$.
\end{proposition}
Let $h\in\omega^\omega$  be as in \Cref{numKpropsolovay}
and suppose (without loss of generality) $h$ dominates $K$. 
Note that for every $m\leq n$,
$K(m)\leq h(m)\leq h(n)$. i.e.,
$$
h(n)\geq \max_{m\leq n}K(m).
$$

\ \\

\noindent\textbf{Realizing (\ref{numKeq53}).}
Now we define a Kraft function $\h K$ (depending on $\rho,b$)
to realize (\ref{numKeq53}).
The goal of this function is:
for every level $n\geq |\rho|$, in the end, 
\begin{align}\label{numKeqtargetV}
\text{there exists some $\sigma_n\in 2^n\cap p$ such that $\h K(\sigma)\leq h(n)+2|\rho|$.}
\end{align}
Since $p$ is not a computable tree, this goal can only be achieved in an online fashion
(if it can be achieve at all).
For $t\in\omega$, let $K[t]$ denote the right c.e. approximation of $K$ at time $t\in\omega$.
Clearly this give rise to approximations of condition $q$ (denoted as $q[t]$),
set $\Delta_{<a}$ (denoted as $\Delta_{<a}[t]$) and $\Delta_K$ (denoted as $\Delta_K[t]$) etc.
More precisely,
the player $\h K$ (playing against $K$) maintain, at each round $t\in\omega$
a sequence $$(\sigma_n[t]\in 2^n\cap p[t]\cap \D_{\geq 3b}[t]: n\geq |\rho|)$$ of strings
such that (for every $n\geq |\rho|$)
\begin{align}\label{numKeq51}
\h K[t](\sigma_n[t])\leq h(n)+2|\rho|.
\end{align}

Note that only in the case of following events, can the working achievement of $\h K$   be sabotaged
\footnote{It is convenient to assume that at each round $t$,
there is only one  $\eta\in\omega\cup 2^{<\omega}$ so that $K[t](\eta)\ne K[t](\eta)$
and it is convenient to   think of $K$
as defined on both integers and strings.
So in the round an integer is compressed, no working achievement is 
sabotaged.}:
\begin{enumerate}
\item Either a $\tau\in p[t-1]$ is kicked out of $p[t-1]$ (at round $t$);
\item Or  $p[t ] = p[t-1]$ but some $\sigma_n[t-1]$  is kicked out of $\D_{\geq 3b}[t-1]$;
\end{enumerate}
We calculate the cost for $\h K$ to repair these sabotage and compare the cost 
to the cost of $K$ (to make the event happening)

\noindent
\textbf{Cost in case of event (1):}
Player $\h K$ need to replace each $\sigma_n[t-1]$ in $[\tau]^\preceq$
(by $\sigma_n[t]$)
and re-assign mass to $\sigma_n[t]$.
The cost of $\h K$ is:
$$
\sum_{\sigma_n[t-1]\in [\tau]^\preceq} 2^{-h(n)-2|\rho|}.
$$
On the other hand, the cost for $K$ (to kick out $\tau$)
is greater than $ 2^{-K[t-1](|\tau|)-b} $.
Therefore, to manage the cost of $\h K$, player 
$\h K$ arrange $(\sigma_n[t]:n\geq |\rho|)$ so that
for every $\tau \in [\rho]^\preceq$
\begin{align}\label{numKeqcost}
\sum_{\sigma_n[t]\in [\tau]^\preceq} 2^{-h(n)-2|\rho|}
\leq 2^{-K[t](|\tau|)-b}.
\end{align}
We will later show that (\ref{numKeqcost}) could indeed be 
arranged.

\noindent
\textbf{Cost in case of event (2):}
Player $\h K$ simply replace $\sigma_n[t-1]$
by an arbitrary string $\sigma\in 2^n\cap p[t]\cap \D_{\geq 3b}[t]$
so that arranging $\sigma_n[t]:=\sigma$
complies to  (\ref{numKeqcost})
(again, we will later show that such arrangement exists).
The cost of $\h K$ is: $2^{-h(n)-2|\rho|}$
while the cost for $K$
(to kick $\sigma_n[t-1]$ out of $ \D_{\geq 3b}[t-1]$) satisfies 
$$2^{-K[t](n)-3b}> 2^{-h(n)-2|\rho|}.$$
The inequality is due to $\rho$ sufficiently long (depending on $b$).

We have shown that in each case of event (1)(2),
the cost of $\h K$ (to re-achieve (\ref{numKeq51})) is smaller than the cost of $K$
(to make the event happening). 
But in case of  event (1), this cost bounding is under hypothesis
that there is an arrangement satisfying  (\ref{numKeqcost}).
We show that (\ref{numKeqcost}) can indeed be arranged.

\ \\


 \noindent\textbf{Re-arrange $(\sigma_n[t]:n\in\omega)$
 ——case of event (1):} When $\tau$ is kicked out of 
$p[t-1]$ (at round  $t$).

In the increasing order of $n\in\omega$  with $\sigma_n[t-1]\in [\tau]^\preceq$,
find an arbitrary $\sigma\in 2^n\cap p[t]\cap \D_{\geq 3b}[t]$ to make
$$
\h K[t](\sigma)\leq h(n)+2|\rho|
$$
without violating (\ref{numKeqcost})
and arrange $\sigma_n[t]:= \sigma$.
We show that 
\begin{claim}
Each $\sigma_n[t-1]\in [\tau]^\preceq$ can be re-arranged.
\end{claim}
\begin{proof}
Suppose otherwise let $n^*$ be the smallest number
so that $\sigma_{n^*}[t-1]\in [\tau]^\preceq$
and there is no $ \sigma\in 2^{n^*}\cap p[t]\cap \D_{\geq 3b}[t]$
so that replacing $\sigma_{n^*}[t-1]$ by $\sigma_{n^*}[t]:=  \sigma$
does not violate (\ref{numKeqcost})
\footnote{ That is, for every $m<n^*$, $\sigma_m[t]$ has been defined.
And, for every $m\in\omega$ with $\sigma_m[t-1]\notin [\tau]^\preceq$,
we arrange $\sigma_m[t]:= \sigma_m[t-1]$.}.

For each $\sigma\in 2^{n^*}\cap [\rho]^\preceq$, let $\eta_\sigma\preceq\sigma$
be the shortest witness for why $\sigma_{n^*}[t]$ cannot be arranged to be $\sigma$.
Clearly there are only three reasons:
\begin{enumerate}

\item   $\eta_\sigma\in \D_{<b}[t] $;
\item $\eta_\sigma = \sigma\in \D_{<3b}[t]$
\item  Arranging so will violate (\ref{numKeqcost}). That is
\footnote{For convenience, we assume 
$$K[0](m)\leq h(m)\wedge (m+O_U(1))$$ for all $m\in\omega$.}:
\begin{align}\label{numKeq52}
\sum_{\sigma_m[t]\in [\eta_\sigma]^\preceq}
2^{-h(m)-2|\rho|}&\geq 2^{-K[t](|\eta_\sigma|)-b}- 2^{-h(n^*)-2|\rho|}\\ \nonumber
&\geq 2^{-K[t](|\eta_\sigma|)-b }/2.
\end{align}
The last inequality is due to $h(n^*)\geq\max_{m\leq n^*} K(m)\geq K(|\eta_\sigma|)$
and $\rho$ sufficiently long (depending on $b$).
\end{enumerate}
Let $W_1, W_2,W_3$ denote the set of witness
of three items respectively.
Note that $W_1\cup W_2\cup W_3$ is prefix free
(since each witness is as short as possible).
Therefore, by lower bound of measure of $[p]$
\footnote{Here we assume that $\h K$ only act at the round $t$ so that $\m([\rho,3b][t]|\rho)>1/2$.},
we have
$$\sum_{\eta\in W_3}2^{-|\eta|} = \m(W_3)\geq \m([\rho,3b]|\rho)\cdot 2^{-|\rho|}\geq  \frac{1}{2}2^{-|\rho|}.$$
Thus, 
\begin{align}\nonumber
\sum_{\eta\in W_3} \sum_{\sigma_m[t]\in [\eta]^\preceq}
2^{-h(m)-2|\rho|}&\overset{(\ref{numKeq52})}{\geq} \sum_{\eta\in W_3} 2^{-K[t](|\eta|)-b }/2\\ \nonumber
&\geq 2^{-O_U(1)}\sum_{\eta\in W_3} 2^{ -|\eta| -b } 
\\ \nonumber
&\geq  2^{-O_U(1)}2^{-|\rho|-b}.
\end{align}
However, since $h$ dominates $K$,
we have
$$
\sum_{\eta\in W_3} \sum_{\sigma_m[t]\in [\eta]^\preceq}
2^{-h(m)-2|\rho|}\leq
 \sum_{m\geq|\rho|} 2^{-h(m)-2|\rho|}\leq 2^{-2|\rho|}< 2^{-O_U(1)}2^{-|\rho|-b}.
$$
The last inequality is due to $\rho$ sufficiently long depending on $U,b$.
Thus a contradiction.

\end{proof}

 \noindent\textbf{Re-arrange $(\sigma_n[t]:n\in\omega)$
 ——case of event (2):} $\sigma_n[t-1]$  is kicked out of $\D_{\geq 3b}[t-1]$.
The proof for this case is identical to that of event (1).

\ \\

In summary, we have shown that 
(\ref{numKeq51}) can be maintained with a cost (by (\ref{numKeqcost})):
$$
\sum_{ \sigma\in 2^{<\omega} } 2^{-\h K(\sigma)}\leq 
\sum_{ \sigma\in 2^{<\omega}}2^{-K(\sigma)}\leq 1.
$$
Therefore, for all $\sigma\in 2^{<\omega}$,
$$
K(\sigma)\leq \h K(\sigma)+O_{\h K}(1)\leq \h K(\sigma)+O_p(1).
$$
The $O_p(1)$ term is due to definition of $\h K$ depending on $p$.
Moreover, we have (\ref{numKeqtargetV}) is achieved so 
$$
\min_{\sigma\in 2^n\cap p} K(\sigma) \leq\min_{\sigma\in 2^n\cap p} \h K(\sigma)+O_p(1)
\overset{(\ref{numKeqtargetV})}{\leq} h(n)+O_p(1).
$$
Thus, for every $y\in \Y_c(p)$,
$$
K(y\uhr n)\leq h(n)+O_p(1)+c.
$$
 But $h(n)+O_p(1)+c$ is also a Solovay function.
 Thus by \Cref{numKth0}, $y$ is K-trivial.
 Thus we are done.

\end{proof}

\numKsimultaneousdrop* 
\begin{proof}
We prove \Cref{numKsimultaneousdrop} in a similar fashion as 
\Cref{numKlemcommonsolutionisKtrivial} by defining a Kraft function  $\h K$
such that
for each $b>O_U(1)$, each $2^b$ many conditions
$p_i= (\rho_i,b)$ (where $i<2^b$ and $\rho_i $ are mutually incomparable) and each $n\in\omega$
\footnote{If this is achieved,
then the conclusion of (\ref{numKsimultaneousdrop})
does follows since given conditions $(p_i=(\rho_i,b):i<b)$
where $\m([p_i]|\rho_i)>1-2^{-4b}$ may not hold, by \Cref{numKclaimext},
we can extend $p_i $ to $q_i = (\h\rho_i, b)$ so that 
$\m([q_i]|\h\rho_i)>1-2^{-4b}$ and applying the conclusion for $(q_i:i<2^b)$. },
   \begin{align}\label{numKeqtargetV2}
  &\text{If $\m([p_i]|\rho_i)>1-2^{-4b}$ for all $i<2^b$,
  then (in the end)}\\ \nonumber
  &\text{there exists an $m>n$ and $\sigma_i\in 2^m\cap p_i$ (for each $i<2^b$)}\\ \nonumber
&\text{such that $\h K(\sigma_i)\leq K(m)+ b$.}
\end{align}

In this proof, 
we call $(b,(\rho_i: i<2^b), n)$
a \emph{tuple} iff $b,n\in \omega$
with $b>O_U(1)$  and $(\rho_i\in 2^{<\omega}:i<2^b)$ are mutually incomparable.
We assign an integer interval $I_{b,\bm\rho,n}$
(hence called \emph{working interval}) to each tuple so that
\begin{enumerate}
\item Intervals assigned to different tuples are mutually disjoint
and the union of working intervals almost covers $\omega$ (except for a finite set);
\item $\min I_{b,\bm\rho,n}$ is sufficiently large depending on 
$b,\bm\rho,n$  and each  working interval   is sufficiently long.
\end{enumerate}
Let $\mcal I$ denote the collection of working intervals.
Player $\h K$ will realize (\ref{numKeqtargetV2}) for each tuple by some $m$ in the working interval $I$ assigned to that tuple
such  that (to manage the cost) $K(m)$ is the maximal among $(K(n):n\in I)$.
For each working interval $I$,
let $m_I[t]\in I$ be such that $K[t](m_I[t])$ is the maximal among 
$(K[t](n): n\in I)$. 
At each round $t$, for each working interval $I$ assigned to, say $b,\bm\rho,n$
\footnote{If this working assumption is not true, $\h K$ needs to do nothing.},
\begin{align}\label{numKeqworkingachieve}
&\text{If $\m( [p_i[t]]|\rho_i)>1-2^{-4b}$ for all $i<2^b$, then player $\h K$ maintains a tuple}\\ \nonumber
&\text{$(\sigma_{I,i}[t]\in 2^{m_I[t]}\cap p_i[t]: i<2^b)$
with $[\sigma_{I,i}[t]]^\preceq\cap \D_{<b}[t]
\subseteq 2^{>m_I[t]+3b}$ }\\ \nonumber
&\text{and assign mass to them such that }\h K[t](\sigma_{I,i}[t])\leq K[t](m_I[t])+b.
\end{align}
Obviously, if player $\h K$ can 
achieve this at infinitely many round $t$, then he could achieve
(\ref{numKeqtargetV2}) in the end.
Also note that, putting cost of $\h K$ aside,
 (\ref{numKeqworkingachieve}) can indeed be maintained.
This only requires to note that 
if $\m([p_i]|\rho_i)>1-2^{-4b}$, then for any $m>|\rho_i|$, there exists 
some $\sigma\in 2^m\cap p_i$ such that 
$[\sigma]^\preceq\cap \D_{<b}
\subseteq 2^{>m+3b}$,
which follows easily since
otherwise $\m([p_i]|\sigma)\leq 1-2^{-3b}$
for all $\sigma\in 2^m\cap p_i$ a contradiction
to $\m([p_i]|\rho_i)>1-2^{-4b}$.

\def\typem{type-$m$}
\def\typetaubelow{type-$\tau$-below}
\def\typetauabove{type-$\tau$-above}
Given a working interval $I$ (say assigned to $b,\bm\rho,n$), there are three ways
(for player $K$) to sabotage the achievement on $I$
\begin{enumerate}
\item Kill $m_I$  i.e., 
\begin{align}\nonumber
&\text{$K[t](m_I[t])\ne K[t-1](m_I[t-1])$ }\\ \nonumber
&\text{or $m_I[t]\ne m_I[t-1]$}.
\end{align}
\item Kill $\sigma_{I,i}$ (for some $i<2^b$) from below. i.e., Kick it out of $p_i $;
\item Kill $\sigma_{I,i}$  (for some $i<2^b$) from above. i.e., $[\sigma_{I,i}]^\preceq\cap \Delta_{<b} 
\cap 2^{\leq m_I+3b}\ne\emptyset$.
\end{enumerate}

Referring to these sabotage as type-$m$, \typetaubelow,
type-$\tau$-above sabotage respectively.
These sabotage are clearly caused by the following three type of events for $I$:
\begin{enumerate}
\item $K[t](m_I[t-1])\ne K[t-1](m_I[t-1])$.
\item For some $\rho_i\preceq\tau\preceq \sigma_{I,i}[t-1]$, we have $K[t](\tau)<K[t-1](\tau)$
and $\tau\in \D_{<b}[t]$.

\item For some $\tau\succeq\sigma_{I,i}$,
we have   $K[t](\tau)<K[t-1](\tau)$ and 
$\tau\in \Delta_{<b}[t]\cap 2^{\leq m_I[t]+3b}$.
\end{enumerate}
Referring these events as 
type-$m$, type-$\tau$-below,
type-$\tau$-above events respectively;
and   call $m=m_I[t-1]$ ($\tau$ respectively) the \emph{witness} of 
the three type of events (for $I$).
Obviously, each type of event will cause the corresponding type of sabotage
and each type of sabotage can only be caused by the corresponding type of event.

Note that in case of \typem\ event, the witness could only sabotage
the working achievement of the working interval containing the witness;
using \Cref{numKclaim4} below we show that this is also the case for \typetauabove\
event;
on other hand, in case of \typetaubelow\ event, the witness could also 
sabotage the working interval above $|\tau|$.
It's not hard to see that:
\begin{remark}\label{numKrem1}
The cost  for $\h K$ to  repair on working interval $I$
with $m\in I$ or $|\tau|\in I$ is bounded by the cost for $K$
to (make the event happening), this is only because
$K(m_I)$ is the maximal among $(K(\h n):\h n\in I)$.
On the other hand, the cost for $\h K$ to repair on working intervals $(J: J>|\tau|)$
is very small (compared to that of $K$)
since each $|I|$ is sufficiently long and $m_I$ much smaller than $\max I$.
\end{remark}

The rest of this proof 
will carry out
the cost   calculation mentioned in \Cref{numKrem1}.
 To this end, we observe:
\begin{claim}\label{numKclaim4}
For every working interval $I$, $m_I$ is much smaller than 
$\max I$. In particular $$\max I-m_I>\min I>3b.$$
So the \typetauabove\ event can only sabotage the 
working achievement of working interval $J$
such that $|\tau|\in J$.
\end{claim}
\begin{proof}
This is simply because $I$ is sufficiently long
and $\mcal I$ is computable, so every $m\in I$ that is close to $\max I$
will have much shorter $U$-description   than most of the elements in $I$.

More precisely, since
$$\sum_{\h n\in I}2^{-K[t](\h n)}\leq 1,$$
by averaging over $I$, we obviously have
\begin{align}\label{numKeq72}
\max_{\h n\in I}K[t](\h n)\geq \log_2 |I|.
\end{align}
On the other hand, among $\mcal I$,
$I$ is at most the $(\min I)^{th}$ interval. Therefore,
since $\mcal I$ is computable, 
for every $\h n\in I$ with $\max I-\h n\leq \min I$,
we have
\footnote{For convenience, we assume 
$K[0](\h n)\leq 2\min I$ for all $\h n\in I $
with $\max I-\h n\leq \min I$ and all working interval $I$
(otherwise player $\h K$ wait for such a moment for $I$ before working on $I$).}
\begin{align}\label{numKeq70}
K(\h n)\leq 2\min I < \log_2 |I|.
\end{align}
The last inequality is due to $|I|$ sufficiently long.
 Thus combining (\ref{numKeq72})(\ref{numKeq70}) 
 we have $\max I-m_I>\min I$.
 The inequality $\min I>3b$ follows by choice of  $I$.
\end{proof}

\noindent
\textbf{Cost of repair for \typem\   event:}
Let $m$ be the witness and $I\ni m$ (so it must be the case $m= m_I[t-1]$).
In this event, player $\h K$ need to re-arrange $\sigma_{I,i}[t-1]$ to other strings
(when $m_I[t]\ne m_I[t-1]$)
and re-assign mass to 
$(\sigma_{I,i}[t]:i<2^b)$.
The cost for $K$ (to make this event) is, trivially $2^{-K[t](m)}$.
If $m_I[t] = m_I[t-1]$, then the cost for $\h K$ is
$$
2^b 2^{-K[t](m)-b}\leq 2^{-K[t](m)}.
$$
If $m_I[t] \ne m_I[t-1]$ (so $K[t](m_I[t])>K[t](m)$),
then the cost for $\h K$ is
$$
2^b 2^{-K[t](m_I[t])-b}< 2^{-K[t](m)}.
$$
Whichever is the case the cost for $\h K$ is bounded by that of cost of $K$.

\noindent\textbf{Cost of repair for \typetaubelow\   event:}
Let $\tau $ be the witness and let working interval $I\ni |\tau|$
(suppose $I$ is assigned to $b,\bm\rho,n$).
Firstly, 
for each affected working interval $J$ (assigned to tuple $  b_J, \bm\rho_J,  n_J$) 
there is at most one $\sigma_{J,i}[t]$ that is affected
(meaning $\rho_{J,i}\preceq\tau\preceq\sigma_{J,i}[t]$  and 
$\tau\notin p_{J,i}[t]$) since $\bm\rho_J$ are mutually incomparable.

For each affected working interval $J$,
re-arrange the affected $\sigma_{J,i}[t]$
  to be (arbitrary) strings as
in (\ref{numKeqworkingachieve})  
(if for some $i<2^{  b_J}$ such string does not exists, then $\h K$ need to do nothing
for interval $J$).
Then re-assign mass to $\sigma_{J,i}[t]$.
There are two subcases.

\textbf{Subcase 1:} Working achievement on $I$ is sabotaged.
The cost for $K$ (to make this event) is:
 $$2^{-K[t](\tau)}\geq 2^{-K[t](|\tau|)-b}\geq 2^{-K[t](m_I)-b}.$$
 The first inequality is due to $\tau$ being a witness for 
 \typetaubelow\ event for $I$
 (so $\tau\notin p_i[t]$). The second inequality is due to $K[t](m_I )$
 is maximal  among $(K(\h n):\h n\in I)$.
Then, the cost for $\h K$ is bounded by:
 \begin{align}\label{numKeq71}
& 2^{-K[t](m_I)-b}+\sum_{J>|\tau|}2^{-K[t](m_J)-b_J}\\ \nonumber
\text{by maximal of $K[t](m_J)$
\ \ \ \ \ }\leq &2^{-K[t](m_I)-b} +\sum_{J>I} \frac{1}{|J|}\\ \nonumber
\leq& 2^{-K[t](m_I)-b}+\frac{1}{|I|^2}\\ \nonumber
\leq& 2\cdot 2^{-K[t](m_I)-b}.
\end{align}
The last two inequality are due to sufficiently long of $J$
and sufficiently long of $I$  depending on $b,\bm\rho,n$ 
(we may assume by choice of $I$: $2^{-K(\h n)}\geq 2/|I|^2$ for all $\h n\in I$
and all working interval $I$).
Thus, the cost for $\h K$ is bounded by  two times that of cost of $K$.

\textbf{Subcase 2:} Working achievement on $I$ is not sabotaged.
The cost for $K$ satisfies:
$$2^{-K[t](\tau)}\geq  2^{-2\max I}.$$
The cost for $\h K$ is (similarly as (\ref{numKeq71})) bounded by:
$$\sum_{J>|\tau|}2^{-K[t](m_J)-b_J}<2^{-2\max I}.$$
The last inequality is also due to $J$ sufficiently long
(so $|J|$ could be sufficiently large depending on $\min J $)

\noindent\textbf{Cost of repair for \typetauabove\   event:}
Let $\tau $ be the witness and working interval $I\ni |\tau|$
(suppose $I$ is assigned to $b,\bm\rho,n$).
Firstly re-arrange the affected $\sigma_{I,i}$ (as the previous case
and note that there is at most one $\sigma_{I,i}$
affected by $\tau$).

The cost for $K$ (to make this event) is:
 $$2^{-K[t](\tau)}\geq 2^{-K[t](|\tau|)-b}\geq 2^{-K[t](m_I)-b}.$$
The cost for $\h K$ is (by \Cref{numKclaim4} only working achievement of $I$ is sabotaged in
the \typetauabove\ way):
$$
2^{-K[t](m_I)-b}
$$
Thus, the cost for $\h K$ is bounded by    that of cost of $K$.

We have shown that in each case the cost of player $\h K$
is bounded by at most two times that of $K$.
Thus the cost of $\h K$ is bounded by, up to an absolute multiplicative constant,
that of $K$. 
Therefore $$K(\sigma)\leq \h K(\sigma)+O_U(1).$$
Thus the conclusion follows by working achievement
(\ref{numKeqtargetV2})
of $\h K$.

\end{proof}

%% file: acknowledgement.tex
\section*{Acknowledgement}
\noindent\textbf{Funding.} Lu Liu is supported by  National Natural Science Foundation of China grant 12471002.

\noindent\textbf{Use of AI.} This research is aided by AI.
The general approach for   \Cref{numKpropfinitenoniosolution},
\Cref{numKpropintersectionoftwo} 
(but not the two propositions themselves) in Section \ref{numKsecsmall},
namely  \Cref{numKlem5},  \Cref{numKlem6}  
(and concepts such as conditions and common solution) are proposed (and proved) by authors
before any use of AI.
We intended to use this approach to study if there is a non i.o. $K$-trivial
so that the $\leq_K$-cone below it is countable. During this trial,
author guessed, without prove, technical 
 \Cref{numKsimultaneousdrop}, which is later proved by AI. 
All other technical lemmas  are proposed and proved by AI
when author raised related technical questions.
 
 \Cref{numKpropslowgrowth}  is discovered by author during conversation with AI
when observing a technical result proposed and proved by AI.
Other main results are proposed by AI.

Author read and checked all results and proofs provided by AI and take full responsibility.
 We also used AI for checking proofs,  literature search
 and polishing.